\documentclass[a4paper,11pt]{amsart}
\usepackage{graphicx}
\usepackage{amssymb}
\usepackage{epstopdf}
\usepackage{nicefrac}
\usepackage{mathrsfs}
\usepackage{amsfonts}
\usepackage{amsthm}
\usepackage{enumerate}

\newtheorem{thm}{Theorem}[section]

\newtheorem{defn}[thm]{Definition}

\newtheorem{lemma}[thm]{Lemma}

\newtheorem{prop}[thm]{Proposition}

\newcommand{\exo}{\Sigma}

\newcommand{\mH}{{\mathbb H}}

\newcommand{\mR}{{\mathbb R}}

\newcommand{\mZ}{{\mathbb Z}}

\newcommand{\cD}{{\mathcal D}}

\newcommand{\cL}{{\mathcal L}}
\newcommand{\cM}{{\mathcal M}}

\newcommand{\Met}{\mathcal{MET}}
\newcommand{\cT}{{\mathcal T}}

\newcommand{\lra}{\longrightarrow}
\newcommand{\bgd}{\begin{displaymath}}
\newcommand{\edd}{\end{displaymath}}

\begin{document}

\title[Teichm\"uller space of Gromov Thurston Manifolds]{Teichm\"uller space of negatively curved metrics on Gromov Thurston Manifolds is not contractible}

%\thanks{2000 {\it Mathematics Subject Classification}. Primary 57R30, 37C55, 37B45; Secondary 53C12 }
 
\author{Gangotryi Sorcar}
\address{Department of Mathematics, SUNY Binghamton, 4400 Vestal Pkwy E  Binghamton, NY 13902}
\email{sorcar-at-math-dot-binghamton-dot-edu}
\thanks{date: September 17, 2013}

\date{}
 

\maketitle

\begin{abstract}
In this paper we prove that for all $n=4k-2$, $k\ge2$ there exists closed $n$-dimensional Riemannian manifolds $M$ with negative sectional curvature that do not have
the homotopy type of a locally symmetric space, such that $\pi_1(\cT^{<0}(M))$ is  non-trivial. $\cT^{<0}(M)$
denotes the Teichm\"uller space of all negatively curved Riemannian metrics on $M$, which is the topological quotient of the space of all 
negatively curved metrics modulo the space of self-diffeomorphisms of $M$ that are homotopic to the identity. Gromov Thurston branched cover manifolds provide examples
of negatively curved manifolds that do not have the homotopy type of a locally symmetric space. These manifolds will be 
used in this paper to prove the above stated result.
\end{abstract}

\tableofcontents

\section{Introduction}\label{sec-intro}

Let $M$ be a closed smooth Riemannian manifold. Let us now look at the following notations we shall use in this paper.

\begin{itemize}
\item $\Met(M)$ denotes the space of all Riemannian metrics on $M$ with smooth topology. [Note that $\Met(M)$ is contractible. 
Any two metrics can be joined by a line segment in the space of metrics, as the convex combination of two metrics is also a metric.]
\item Diff$(M)$ is the group of all smooth self-diffeomorphisms of $M$. $\text{Diff}(M)$ acts on $\Met(M)$ by pushing forward metrics, 
that is, for any $f$ in Diff$(M)$ and any metric $g$ in $\Met(M)$, $f_*g$ is the metric such that $f:(M,g)\rightarrow(M,f_*g)$ is an isometry.
\item $\text{Diff}_0(M)$ stands for the subgroup of Diff$(M)$ consisting of all smooth self diffeomorphisms of the manifold $M$ 
that are homotopic to the identity. 
\item $\cD_0(M)$ is the group $\mR^+\times$ Diff$_0(M)$. 

\item $\cD_0(M)$ acts on $\Met(M)$ by scaling and pushing forward metrics, that is, when $(\lambda,f) \in \cD_0(M)$ and $g \in$ $\Met(M)$, 
$(\lambda,f)g=\lambda(f)_*g$. The Teichm\"uller space of all metrics on $M$ is defined to be $\cT(M):=\Met(M)/\cD_0(M)$.
\item Similarly, the Teichm\"uller space of negatively curved metrics on $M$ is defined as $\cT^{<0}(M):=\Met^{<0}(M)/\cD_0(M)$.
\end{itemize}

Since a geometric property of a manifold is often thought of as one which is invariant under isometries and scalings, two Riemannian metrics being equal in 
$\cT(M)$ implies that they impart the same ``marked geometry'' on the manifold. More on this can be found in \cite{F}.

{\it{Remark.}} If $M$ is an orientable surface of genus $g >1$, the Teichm\"uller space of $M$ is defined to be the quotient of the space 
of all hyperbolic metrics by the action of the group Diff$_0(M)$, 
but it is not effective to generalize it for manifolds with dimension $\ge 3$ as Mostow's rigidity theorem implies that, in dimensions $3$ or more, 
this space contains (at most) one point.

Studying these spaces is of fundamental importance in geometry and also other areas, one could see for example \cite{BES} on Einstein Manifolds
or \cite{FO2} on fiber bundles with negatively curved fibers. \newline

In this paper we prove:
\begin{thm}\label{mainthm}
For every positive integer $n=4k-2$ where $k$ is an integer more than $1$, there is a negatively curved manifold $M^n$, that does not have the homotopy type of 
a locally symmetric space, such that $\pi_1(\cT^{<0}(M^n)) \ne 0$. Therefore $\cT^{<0}(M)$ is not contractible. 
\end{thm}

{\bf {Idea of proof:}} Suppose M is any manifold. Consider the sequence:\\
\bgd
 \cD_0(M) \lra \Met^{<0}(M) \lra \frac{\Met^{<0}(M)}{\cD_0(M)}=:\cT^{<0}(M).
\edd

By the work of Borel, Conner and Raymond \cite{BCR} one gets that $\cD_0(M)$ acts freely on $\Met(M)$ and more details on this can be found in page 51 of 
\cite{FO1}. Then by using Ebin's Slice Theorem one can deduce that the above sequence is a fibration.

Hence from the above fibration we get a long exact sequence in homotopy, part of which we use and is shown below:
\bgd
\pi_1(\Met^{<0}(M)) \lra \pi_1(\cT^{<0}(M)) \lra \pi_0(\cD_0(M)) \lra \pi_0(\Met^{<0}(M)).
\edd

We want to come up with a $y \ne 0$ in $\pi_1(\cT^{<0}M)$. To this end, we shall first find $[f] \ne 0$ in $\pi_0(\cD_0(M))$ such that
it maps to zero in $\pi_0(\Met^{<0}(M))$. Therefore $[f]$ can be pulled back in $\pi_1(\cT^{<0}(M))$, and since the pullback of a non-zero element cannot be zero,
this pullback serves as the $y\ne0$ in $\pi_1(\cT^{<0}(M))$ we are looking for. 

In section \ref{sec-proof} when we prove this theorem we shall construct this $[f]\in \pi_0(\text{Diff}_0(M))$ and follow through with the above idea.
This is fine because $\text{Diff}_0(M)$ is a subset of $\cD_0(M)$.

Also, we work with manifolds that are negatively curved but not of the homotopy type of locally symmetric spaces and for this we will
use the Gromov-Thurston manifolds as described in the following section.

\section{Gromov-Thurston Manifolds}\label{sec-GTmani}

In 1987, in their paper \cite{G-T}, Gromov and Thurston introduced examples of manifolds that cannot be hyperbolic but support metrics of negative curvature.
In fact these manifolds do not have the homotopy type of locally symmetric spaces. They explicitly construct these manifolds by taking branched covers over 
a codimension two totally geodesic submanifold of a closed real hyperbolic manifold. 

Let $N^n$ be a closed, orientable, hyperbolic manifold such that $V^{n-2} \subset W^{n-1} \subset N$ with $\partial W = V$, where $V$ and $W$ are oriented 
totally geodesic submanifolds of $N$. $N$ is cut along $W$ to produce a new manifold $\overline{N}$
that has boundary the union of two copies of $W$, i.e. $W \sqcup (-W)$. We write $W$ and $-W$ to denote that the orientation on $N$ induces 
opposite orientations in the two boundary components. Fixing an integer $s$ we can take $s$ copies of $\overline{N}$ and paste them along their boundaries 
$W \sqcup (-W)$ in a way compatible with their orientations. 
In this way we obtain a manifold $N_s$ and a map $\rho:N_s\rightarrow N$ that is a branched cover of $N$ along $V$.

$N_s$ can be given a Riemannian metric that has negative curvature which is of constant negative curvature outside a tubular neighborhood of the branching set $V$.
The detailed construction of this can be found in \cite{G-T} and also in the second chapter of Ardanza's thesis \cite{ARD}. We will quote some defintions
and results we use from the latter:

\begin{defn}
 Given $V$ compact totally geodesic submanifold of a Riemannian manifold $N$ we denote by $Rad^\perp(V)$ the normal injectivity radius of $V \subset N$ that is the
 greatest real number $r$ for which the open neighborhood $U_r\subset N$ of $V$ is diffeomorphic to the normal bundle of $V$ in $N$.
\end{defn}

\begin{prop}\label{prop:norinjrad}
Given $N^n$ hyperbolic manifold, $V^{n-2}$ totally geodesic submanifold that bounds in $N$ with $Rad^\perp(V) \ge R$, and $\rho:N_s\rightarrow N$, a branched
cover of $N$ as described before, there is a metric $g_s$ on $N_s$ with negative sectional curvature $K$ such that $K\equiv -1$ outside a tubular neighborhood of 
radius $R$ of the branching set $V$. 
\end{prop}

In \cite{G-T} Gromov and Thurston also talk about certain special closed real hyperbolic manifolds that they call the $\phi_n$-manifolds.
The $\phi_n$-manifolds are closed hyperbolic manifolds with large injectivity radius. They are constructed by taking quotients of the hyperbolic $n$-space $\mH^n$ 
by a discrete, cocompact, torsion free subgroup of the isometries of $\mH^n$.

A result from \cite{G-T} about these $\phi_n$ manifolds is:
\begin{thm}\label{thm:thm1}
 For every $n>2$ and every $R>0$ there exists an orientable $\phi_n$-manifold $N$ and a totally geodesic oriented submanifold $V$ of codimension $2$ in $N$
 such that:\\[-0.6cm]
 \begin{enumerate}[(i)]
  \item $Rad^\perp(V) \ge R$
  \item $[V]=0 \in H_{n-2}(N)$; moreover $V$ bounds a connected oriented totally geodesic hypersurface $W \subset N$.
 \end{enumerate}
\end{thm}

The above theorem ensures that the topological construction of the branched coverings can be performed on these $\phi_n$-manifolds. In \cite{G-T} they prove
that certain ones among these branched covers do not have the homotopy type of locally symmetric manifolds, hence in particular cannot admit a metric
of constant sectional curvature $-1$.

\vspace{5mm}

From \cite{ARD} we have the following:
\begin{thm}\label{thm:injrad}
 Let $N$ be a $\phi_n$-manifold and $V$ be a totally geodesic oriented submanifold of codimension $2$ in $N$, such that
 injectivity radius $Rad(N)\ge10\alpha$ for some fixed $\alpha > 0$, and normal injectivity radius $Rad^\perp(V)\ge 5\alpha$.
 Then there is a metric $g_s$ on the branched cover $N_s$ such that if $K(g_s)$ is the sectional curvature then:\\[-0.6cm]
 \begin{enumerate}[(i)]
  \item $K(g_s)<0$ in a tubular neighborhood $T_R(V)$ of radius $R<\alpha$.
  \item $K(g_s)\equiv-1$ in $N_s-T_R(V)$
  \item There exists $p \in N_s - T_R(V)$ such that $d(p,V)>3\alpha$. 
 \end{enumerate}
\end{thm}

%The proof is a consequence of Proposition \ref{prop:norinjrad} and Theorem \ref{thm:thm1}.

The manifold $M$ that we shall use in this paper is $N_s$ with the metric $g_s$ as described in above theorem, when $N$ is one of those $\phi_n$-manifolds 
that do not have the homotopy type of a locally symmetric manifold. 
%\newpage
\section{Proof of theorem}\label{sec-proof}

In this section we shall prove Theorem \ref{mainthm}. 

Let $M^n$ be the branched cover manifold as desribed above at the end of section $2$, with dimension $n=4k-2$, where $k$ is some integer greater than $1$.

First, we want to construct a non-zero element in $\pi_0(\text{Diff}_0(M))$. 
Such an element is represented by a self-diffeomorphism $f$ of $M$ that is homotopic to the identity map on $M$, such that there is no path in 
$\text{Diff}_0(M)$ connecting $f$ to the identity map of $M$. Actually what we end up doing is more than this. 
We construct an $f \in \text{Diff}_0(M)$ such that there is no path in $\text{Diff}(M)$ connecting this $f$ with the identity map on $M$.
If such a path does not exist in $\text{Diff}(M)$ it cannot exist in $\text{Diff}_0(M)$ which is a subspace of $\text{Diff}(M)$.

\subsection{Construction of the diffeomorphism}\label{subsec:construct}
An exotic $n$-sphere $\Sigma^n$ is an $n$-dimensional smooth manifold that is homeomorphic to $S^n$ but not diffeomorphic to $S^n$. $\Sigma^n$
is a twisted double of two copies of $D^n$ joined by an orientation preserving diffeomorphism of the boundary of $D^n$ denoted by $\partial D^n =
S^{n-1}$. Let this orientation preserving diffeomorphism be $h_1:S^{n-1} \rightarrow S^{n-1}$. 

Let $P_1$ and $P_2$ be two distinct points in $S^{n-1}$ and let $N(P_1)$ and $N(P_2)$ be some neighborhood of $P_1$ and $P_2$ respectively in $S^{n-1}$ .
The diffeomorphism $h_1$ is smoothly isotopic to a 
diffeomorphism $h_2:S^{n-1} \rightarrow S^{n-1}$  such that $h_2$ is the identity restricted to the neighborhoods $N(P_1)$ and $N(P_2)$ and is homotopic 
to the identity map on $S^{n-1}$ relative to $N(P_1)$ and $N(P_2)$. This can be achieved with the help of techniques from Differential
Topology in Milnor's Topology from a Differentiable Viewpoint \cite{MIL-BK}.

Let $N'(P_1)$ and $N'(P_2)$ be proper subsets of $N(P_1)$ and $N(P_2)$ respectively. We shall now look at $h_2$ restricted to 
$S^{n-1} \smallsetminus (N'(P_1) \cup N'(P_2))$, and we shall identify $S^{n-1} \smallsetminus (N'(P_1) \cup N'(P_2))$ with $S^{n-2}\times [1,2]$.

By the Pseudoisotopy Theorem of Cerf, this diffeomorphism $h_2$ restricted to $S^{n-2}\times [1,2]$ is smoothly isotopic to a self-diffeomorphism 
$h:S^{n-2}\times [1,2]\rightarrow S^{n-2}\times [1,2]$ such that, $h$ is level preserving, in other words, if we fix any 
$t \in [0,1]$ then for all $x \in S^{n-2}$, $h(x,t)=(y,t)$ for some $y \in S^{n-2}$. We shall denote this $y$ as $h_t(x)$.

The reason for taking smaller neighborhoods inside the original neighborhoods of the two points $P_1$ and $P_2$ is because we want the
self-diffeomorphism $h$ of $S^{n-2}\times [1,2]$ to be the identity near $1$ and $2$, which is desirable due to technical reasons that become clear later.

If we select an exotic sphere $\Sigma$ of dimension $4k-1$, we can get a similar diffeomorphism $h:S^{4k-3}\times [1,2]\rightarrow S^{4k-3}\times [1,2]$ 
by the above process. With the help of this diffeomorphism $h$ we shall construct a diffeomorphism 
$f:M\rightarrow M$ such that $[f] \in \pi_0(\text{Diff}_0(M))$.

At this point let us recall theorem \ref{thm:injrad} and remember that our $M^{4k-2}=N_s$. Let us choose a point $p \in N_s-T_R(V)$, so that by the same theorem,
the injectivity radius of $p$ is greater than $3\alpha$. Also note that $p$ is in the ``hyperbolic region'' of $M$.

Let $B$ be a closed geodesic balls of raidius $2\alpha$ centered at $p$. We will identify $B \setminus p$ with $S^{4k-3} \times (0,2\alpha]$, where the
lines $t \mapsto (x,t)$ are the unit speed geodesics emanating from $p$. 

We will now define $f \in \text{Diff}(M)$ as follows: \\\\
\centerline {$\displaystyle f(q)  = \begin{cases} 
      q & \text{if } q \notin (S^{4k-3}\times[\alpha,2\alpha])\subset M \\
      (h_{t/\alpha}(x),t) & \text{if } q=(x,t) \in (S^{4k-3}\times[\alpha,2\alpha]) \subset M \\
     \end{cases}$}\\\\
where $h(x,s)=(h_s(x),s)$.\newline

The diffeomorphism $h$ is homotopic to the identity map on $S^{4k-3}\times [1,2]$ relative to the boundary, because
$h_2$ is homotopic to the identity map of $S^{4k-2}$ relative to the neighborhoods $N(P_1)$ and $N(P_2)$. Therefore $f$ constructed as above is
homotopic to the identity on $M$. Hence $f \in \text{Diff}_0(M)$. 

\subsection{Showing $f$ is not smoothly isotopic to the identity} We will now show that there is no path in Diff$(M)$ connecting $f$ as above and the identity map 
on $M$. In other words we shall show that $f$ is not smoothly isotopic to the identity on $M$. The definition of ``smooth isotopy'' that we use is 
given below and $I$ denotes the closed interval $[0,1]$ in this definition and for the rest of the paper:\newline

\begin{defn}
 A diffeomorphism $f:M\rightarrow M$ is said to be {\bf{smoothly isotopic}} to another diffeomorphism $g:M\rightarrow M$ if there exists a smooth
 map $F:M\times I \rightarrow M$ such that $F(m,0)=f(m)$ and $F(m,1)=g(m)$ for all $m\in M$, and for each $t_0 \in I$, $h_{t_0}:M\rightarrow M$ is a diffeomorphism,
 where $h_{t_0}(m) := F(m,t_0)$ for all $m \in M$.
\end{defn}

Another definition we shall use in this paper:

\begin{defn}
 Given a manifold $W$ with boundary $\partial_1 W \sqcup \partial_2 W$ and a diffeomorphism $F:\partial_1 W \rightarrow \partial_2 W$, we obtain a manifold
 without boundary, $W_F := W/x\sim F(x)$ where $x \in \partial_1 W$.
\end{defn}
Note:\begin{enumerate}
      \item If we have $W$ as above and a homeomorphism $F:\partial_1 W \rightarrow \partial_2 W$ then $W_F := W/x\sim F(x)$ where $x \in \partial_1 W$,
      is a topological manifold without boundary.
      \item If we have a homeomorphism or a diffeomorphism $f:N\rightarrow N$ then $(N\times I)_f$ is just the usual mapping torus.
     \end{enumerate}

\vspace{5mm}
Two lemmas that we shall use are as follows:
\begin{lemma}\label{lem:lem1}
 If $\alpha$ and $\beta$ are two homotopic self-homeomorphisms of a non positively curved close manifold $N$ then $\alpha$ and $\beta$ are topologically 
 pseudo-isotopic provided the dimension of $N$ is greater or equal to $4$.
\end{lemma}

\begin{lemma}\label{lem:lem2}
 If $\alpha$ and $\beta$ are two topologically pseudo isotopic homeomorphisms of a closed manifold $N$ then $(N\times I)_\alpha$ is homeomorphic to 
 $(N\times I)_\beta$.
\end{lemma}

The second lemma is not difficult to prove. The first lemma is a result by Farrell and Jones \cite{FJ1}.

In all that follows now $\Sigma$ stands for the exotic sphere of dimension $4k-1$ used in subsection \ref{subsec:construct}, and $id_N$ will denote
the identity map on a manifold $N$.\newline

\begin{prop}
There is a diffeomorphism $\mathcal{F}_1:(M\times S^1)\# \exo\rightarrow(M\times I)_f$.
\end{prop}

This is true because of the way we construct $f$.

\begin{prop}\label{prop:pi}
 If $f$ is smoothly isotopic to $id_M$ then, there is a diffeomorphism $\mathcal{F}_2:(M\times I)_f\rightarrow M\times S^1$.
 Also, the induced map ${\mathcal{F}_2}_*:\pi_1((M\times I)_f)\rightarrow \pi_1(M\times S^1)$ restricted to 
 $\pi_1(M) \subset \pi_1((M\times I)_f)$ is the identity map onto $\pi_1(M) \subset \pi_1(M\times S^1)$.
\end{prop}

This can be deduced using the Lemmas 3.3 and 3.4.

Composition of the two diffeomorphisms in the above two propositions yield a diffeomorphism $\mathcal{F}=\mathcal{F}_2\circ\mathcal{F}_1$: 
\begin{displaymath}
 \mathcal{F}:(M\times S^1)\# \exo\rightarrow (M\times S^1)
\end{displaymath}

Let $T$ denote $M\times S^1$ from now onwards.

Assuming this diffeomorphism $\mathcal{F}:T\# \exo\rightarrow T$ exists we want to arrive at a contradiction. 

Let $\exo$ we selected bound $W$, a $4k$-dimensional parallelizable manifold. In Kervaire and Milnor's paper \cite{MIL-KER}
we learn that such exotic spheres that bound parallelizable manifolds do exist in the dimensions that we are working with. We shall use information
on these special exotic spheres as and when we need from \cite{MIL-KER}.

Let us now denote the boundary connect sum of $T \times [0,1]$ and $W$ by $T\times [0,1] \hspace{1mm} \#_b  \hspace{1mm} W$.
The boundary of this manifold $T\times [0,1] \hspace{1mm} \#_b  \hspace{1mm} W$ is the disjoint union of $T\hspace{1mm}\#\hspace{1mm}\exo$ and $T$. 

We now build a new closed, smooth, orientable manifold $\mathcal{M}$ of dimension $4k$ as described below:
\begin{defn} $\mathcal{M}:=(T \times [0,1] \hspace{1mm} \#_b  \hspace{1mm} W)_\mathcal{F}$.
\end{defn}

Let $D$ be a $4k$ dimensional disk. The boundary of $D$ is $S^{4k-1}$. The boundary of $W$ is $\exo$ which is homeomorphic to $S^{4k-1}$. 
By gluing the boundaries of $W$ and $D$ via this homeomorphism we obtain a compact topological manifold without boundary, which we will 
denote as $W\cup D$.

\begin{prop}\label{prop1}
$\mathcal{M}$ is homeomorphic to $(T\times S^1)\hspace{1mm}\#\hspace{1mm}(W \cup D)$, where $D$ is a $4k$-dimensional disk.
\end{prop}

\begin{proof}
 Let us take the connect sum of $T\times[0,1]$ by removing a point from its interior with the topological manifold $W\cup D$ and denote
 the resulting manifold with boundary by $\overline{W}$. We write:
 \begin{displaymath}
  \overline{W}=(T\times [0,1])\hspace{1mm}\#\hspace{1mm}(W \cup D)
 \end{displaymath}
 Note that $\overline{W}$ has two boundary components, both being the manifold $T$ and that $\overline{W}$ is homeomorphic to  
 $(T\times [0,1])\hspace{1mm}\#_b\hspace{1mm}W$. Note that there is a canonical homeomorphism between $T\hspace{1mm}\#\hspace{1mm}\exo$ and $T$.
 
 By Lemma \ref{lem:lem1} and the condition on homotopy in Proposition \ref{prop:pi} we obtain that $\mathcal{F}$ is topologically 
 pseudo-isotopic to the identity map.
 
 Now by Lemma \ref{lem:lem2} and the fact that $\overline{W}$ is homeomorphic to  
 $(T\times [0,1])\hspace{1mm}\#_b\hspace{1mm}W$, we can conclude that $\mathcal{M}:=(T \times [0,1] \hspace{1mm} \#_b  \hspace{1mm} W)_\mathcal{F}$
 is homeomorphic to $(\overline{W})_{id}$. 
 
 Let us note that:
\begin{displaymath}
 (\overline{W})_{id} = (T\times S^1)\hspace{1mm}\#\hspace{1mm}(W \cup D)
\end{displaymath}
 where $id:T\rightarrow T$ is the identity map. \\
 
Therefore we have proved that $\mathcal{M}$ is homeomorphic to $(T\times S^1)\hspace{1mm}\#\hspace{1mm}(W \cup D)$.
 
\end{proof}

For any manifold $N$ let $\sigma(N)$ denote the signature of the manifold.

By the proposition we proved above, for our $4k$-dimensional manifold $\cM$,
\begin{displaymath}
 \sigma(\cM)=\sigma((T\times S^1)\hspace{1mm}\#\hspace{1mm}(W \cup D))
\end{displaymath}
since the signatures of two homeomorphic manifolds are equal. 
By the properties of the signature we get, 
\begin{displaymath}
 \sigma((T\times S^1)\hspace{1mm}\#\hspace{1mm}(W \cup D))=\sigma(T\times S^1)+\sigma(W \cup D)
\end{displaymath}
\begin{displaymath}
 =\sigma(T)\sigma(S^1)+\sigma(W \cup D)
\end{displaymath}
\begin{displaymath}
 =\sigma(W\cup D)
\end{displaymath}

Also by definition:
\begin{displaymath}
 \sigma(W\cup D)=\sigma(W)
\end{displaymath}

Therefore
\begin{displaymath}
 \sigma(\cM)=\sigma(W)
\end{displaymath}

On the other hand by the Hirzebruch index theorem,
\begin{displaymath}
 \sigma(\cM)=\cL[\cM]= \langle \cL_k(p_1,\dots,p_2),\mu_{4k}\rangle
\end{displaymath}
where $\cL[\cM]$ is the $\cL$-genus of $\cM$.
$\cL_k$ is a polynomial in the first $k$ Pontryagin classes of $\cM$ with rational coeffcients.

The $i$-th Pontryagin class of the manifold $\cM$ is an element in $H^{4i}(\cM,\mZ)$. Using the fact that $\mathcal{M}$ is homeomorphic to 
$(T\times S^1)\hspace{1mm}\#\hspace{1mm}(W \cup D)$, and the Mayer Vietoris sequence in cohomology we get the exact sequence:
\bgd
 \cdots \lra H^{i-1}(S^{4k-1}\times I)\lra H^i(\cM) \lra H^i(T\times S^1 \hspace{1mm} \smallsetminus \hspace{1mm} p_1)\oplus H^i(W\cup D \hspace{1mm} 
 \smallsetminus \hspace{1mm} p_2)  \lra H^i(S^{4k-1}\times I) \lra \cdots .
\edd

where $p_1$ is a point in $T\times S^1$ and $p_2$ is a point in $W\cup D$. 
Since,
\begin{displaymath}
 H^{i-1}(S^{4k-1}\times I)=H^i(S^{4k-1}\times I)=0\,\,\,\text{for}\,\,\, 2<i<4k-1
\end{displaymath}
we have,
\begin{displaymath}
 H^i(\cM)\cong H^i(T\times S^1 \smallsetminus p_1)\oplus H^i(W\cup D \smallsetminus p_2)\,\,\,\text{for}\,\,\, 2<i<4k-1
\end{displaymath}
It is not difficult to see that 
\begin{displaymath}
 H^i(T\times S^1 \smallsetminus p_1)=H^i(T\times S^1)\hspace{2mm}  \text{and}\hspace{2mm} H^i(W\cup D\smallsetminus p_2)=H^1(W\cup D) 
 \hspace{2mm}\text{for}\hspace{2mm} 2<i<4k-1
\end{displaymath}
Therefore, we have
\begin{displaymath}
 H^i(\cM)\cong H^i(T\times S^1)\oplus H^i(W\cup D)\,\,\,\text{for}\,\,\, 2<i<4k-1
\end{displaymath}

Therefore by Novikov's Theorem for rational Pontryagin classes, 
\begin{displaymath}
 p_j(\cM)\mapsto (p_j(T\times S^1),p_j(W\cup D))\,\,\,\text{for}\,\,\, 1\le j \le k-1
\end{displaymath}
Since $W$ is parallelizable, $P(W)$ the total Pontryagin class of $W$, is trivial. Also in chapter 4 of his thesis \cite{ARD} Ardanza proves that 
$P(M)$ the total Pontryagin class of $M$ (where $M$ is the Gromov Thurston manifold we are using) is trivial. Let us recall that $T=M\times S^1$, therefore,
$P(T\times S^1)=P(M)P(S^1\times S^1)$ is also trivial.

Hence,
\begin{displaymath}
 p_j(\cM)=0\,\,\, \text{for all}\,\,\, 1\le j \le k-1
\end{displaymath}

So the only surviving Pontryagin class in the $\cL_k$-polynomial of $\cM$ is $p_k$. The coeffcient of $s_k$ of $p_k$ in $\cL_k$ is given in page 12 of \cite{HIR} by:
\begin{displaymath}
 s_k=\displaystyle{\frac{2^{2k}(2^{2k-1}-1)}{(2k)!}B_k}
\end{displaymath}
where $B_k$ is the $k$-th Bernoulli number.

Summarizing the above work,
\begin{equation}\label{1}
 \sigma(\cM)=\sigma(W)=\displaystyle{\displaystyle{\frac{2^{2k}(2^{2k-1}-1)}{(2k)!}}B_k \langle p_k,\mu_{4k}\rangle}
\end{equation}

Let us now state a few things about $W$ and $\sigma(W)$, which are stated and proved in \cite{MIL-KER}.

Let $\Theta^n$ be the abelian group of oriented diffeomorphism class of $n$-dimensional homotopy spheres with the connected sum operation. Then, the oriented
diffeomorphism classes of $n$-dimensional homotopy spheres that bound a parallelizable manifold form a subgroup of $\Theta^n$ and will be denoted by 
$\Theta^n(\partial\pi)$.

The following theorems from \cite{MIL-KER} shall be used shortly:
\begin{thm}
 The group $\Theta^n(\partial\pi)$ is a finite cyclic group. Moreover if $n=4k-1$, then 
 \begin{displaymath}
  \Theta^n(\partial\pi)=\mZ_{t_k} 
 \end{displaymath}
where $t_k=a_k2^{2k-2}(2^{2k-1}-1)num\displaystyle{\left(\frac{B_k}{4k}\right)}$ and $a_k=1$ or $2$ if $k$ is even or odd.
\end{thm}

\begin{thm}\label{thm:milker}
 Let $t_k$ be the order of $\Theta^{4k-1}(\partial\pi)$. There exists a complete set of representatives 
 $\displaystyle\{S^{4k-1}=\Sigma_0,\Sigma_1, \dots, \Sigma_{t_k-1}\}$ in $\Theta^{4k-1}(\partial\pi)$ such that if $W$ is a parallelizable manifold with 
 boundary $\Sigma_i$ then $\displaystyle{\left(\frac{\sigma(W)}{8}\right)mod\, t_k=i}$.
\end{thm}

Therfore if we choose some $\Sigma_i$ from the set or representatives in the above theorem, with $\partial W=\Sigma_i$, then
\begin{equation}\label{2}
 \sigma(W)=8(t_kd+i) \,\,\,\,\, \text{for some postive integer}\,\,\,d.
\end{equation}

Rewriting equation (\ref{1}) using the equation (\ref{2}) we get
\begin{equation}
 8(t_kd+i)=\displaystyle{\displaystyle{\frac{2^{2k}(2^{2k-1}-1)}{(2k)!}}B_k \langle p_k,\mu_{4k}\rangle}.
\end{equation}

Now we state another result from chapter 2 of Ardanza's thesis \cite{ARD}:
\begin{thm}
 For all $k > 1$, there exists a prime $p > 2k+1$ such that $p$ divides the numerator of $\displaystyle{\frac{2^{2k}(2^{2k-1}-1)}{(2k)!}B_k}$ expressed in lowest 
 terms.
\end{thm}

It is easy to see that such a prime $p$ has to divide $t_k$.

In equation (3), $\langle  p_k,\mu_{4k}\rangle$ is an integer, and the prime $p$ divides $t_k$ and $\displaystyle{\frac{2^{2k}(2^{2k-1}-1)}{(2k)!}}B_k$
so it has to divide $i$. But we have the freedom of choosing $\Sigma_i$ from the set of representatives in Theorem \ref{thm:milker}
such that $i$ is not divisible by $p$. We could always pick $i=1$ and that will work.

Summarizing the work done so far, we can always choose a $\Sigma_i$ and use the methods of subsection \ref{subsec:construct} to get a self-diffeomorphism 
$f:M\rightarrow M$ such that $[f]\in \pi_0(\text{Diff}_0(M))$ and $[f] \ne 0$ in $\pi_0(\text{Diff}(M))$.

\subsection{Showing that [f] maps to 0 in $\pi_0(\Met^{<0}(M))$}\label{subsec:met}
The diffeomorphism $f$ maps to the metric $fg_s$ in $\Met^{<0}(M)$ under the push-forward action described earlier. The metric $g_s$ is the negatively curved
metric that $M$ comes equipped with from Section \ref{sec-GTmani}. 

In this subsection we show that $fg_s$ can be joined to the metric $g_s$ by a path of metrics in $\Met^{<0}(M)$.

Denote by $B\subset M$ the closed geodesic ball centered at $p$ of radius $3\alpha$. 
Let us also denote the metric $g_s$ on $M$ restricted to $B$ as $g^0$, and note that $g^0$ is the hyperbolic metric. 
As before we identify $B\smallsetminus{p}$ with $S^{4k-3}\times(0,3\alpha]$. This idenification can be done isometrically: $B\smallsetminus{p}$ with metric
$g^0$ is isometric to $S^{4k-3}\times(0,3\alpha]$ with metric $\sinh^2(t)\overline{h}+dt^2$, where $\overline{h}$ is the Riemannian metric on the 
sphere $S^{4k-3}$ with constant curvature equal to $1$. In view of this identification, we write
\begin{displaymath}
 g^0(x,t)=\sinh^2(t)\overline{h}+dt^2.
\end{displaymath}

Recall that,\newline
\centerline {$\displaystyle f(q)  = \begin{cases} 
      q & \text{if } q \notin (S^{4k-3}\times[\alpha,2\alpha])\subset M \\
      (h_{t/\alpha}(x),t) & \text{if } q=(x,t) \in (S^{4k-3}\times[\alpha,2\alpha]) \subset M \\
     \end{cases}$}\\\\
where $h(x,s)=(h_s(x),s)$.\newline
Also recall that, $h:S^{4k-3}\times [1,2]\rightarrow S^{4k-3}\times [1,2]$ is level preserving, and identity near $1$ and $2$.

Therefore, the metric $g^1=f_*g^0$ (the push forward of $g^0$ by $f$) on $B\smallsetminus {p}$ is given by:
\begin{displaymath}
  {\displaystyle g^1(x,t)  = \begin{cases} 
      g^0(x,t) & \text{if } t \notin [\alpha,2\alpha] \\
      h_*g^0(x,t) & \text{if } t \in [\alpha,2\alpha]. \\
     \end{cases}}\\\
\end{displaymath}

Let us state a lemma from \cite{FO1}.
\begin{lemma}
 Let $\mathscr G' \subset \text{Diff}_0(S^{n-1}\times [1,2])$ be the group of all smooth isotopies $h$ of the $(n-1)$-dimensional sphere $S^{n-1}$ that are the 
 identity near $1$ and constant near $2$. Then $\mathscr G'$ is contractible.
\end{lemma}

The diffeomorphism $h:S^{4k-3}\rightarrow S^{4k-3}$ that we use for the contruction of our $f$ is an element of $\mathscr G'$.
Therefore by the above Lemma, there is a path of isotopies $h^\mu \in \mathscr G'$ for $\mu\in [0,1]$, with $h^0=h$ and $h^1= id_{S^{4k-3}\times [1,2]}$. 
Each $h^\mu$ is a smooth isotopy of the sphere $S^{4k-3}$. Let us denote the final map in the isotopy $h^\mu$ as $\theta^\mu$, that is, 
\begin{displaymath}
 h^\mu(x,2)=(\theta^\mu(x),2)
\end{displaymath}
Note that $\theta^\mu:S^{4k-3}\rightarrow S^{4k-3}$ is a diffeomorphism and $\theta^0=\theta^1=id_{S^{4k-3}}$.
Define
\begin{displaymath}
 \phi^\mu:S^{4k-3}\times[\alpha,2\alpha]\rightarrow S^{4k-3}\times[\alpha,2\alpha]
\end{displaymath}
by rescaling $h$ to the interval $[\alpha,2\alpha]$, that is, $\phi^\mu(x,t)=(h^\mu_{t/\alpha}(x),t)$.
Let $\delta:[2,3]\rightarrow [0,1]$ be smooth with $\delta(2)=1$,$\delta(3)=0$, and $\delta$ constant near $2$ and $3$. Now we are in a postion to define
a path of negatively curved metrics $g^\mu$ on $B\smallsetminus {p}=S^{4k-3}\times(0,3\alpha]$:

\begin{displaymath}
  {\displaystyle g^\mu(x,t)  = \begin{cases} 
      g^0(x,t) & \text{if } t \in (0,\alpha] \\
      (\phi^\mu)_*g^0(x,t) & \text{if } t \in [\alpha,2\alpha] \\
      \sinh^2(t)(\delta(\frac{t}{\alpha})(\theta^\mu)_*\overline{h}(x)+(1-\delta(\frac{t}{\alpha}))\overline{h}(x)) + dt^2 & \text{if } t\in [2\alpha,3\alpha]\\
     \end{cases}}\\\
\end{displaymath}

Since $\delta$ and all isotopies used are constant near the endpoints of the intervals on which they are defined, it is clear that $g^\mu$ is a smooth metric
on $B\smallsetminus {p}$ and that $g^\mu$ joins $g^1$ to $g^0$. Moreover $g^\mu(x,t)=g^0(x,t)$ for $t$ near $0$ and $3$. Hence we can extend $g^\mu$ to the whole 
manifold $M$ by defining $g^\mu(q)=g^0(q)$ for $q=p$ and $g^\mu(q)=g_s(q)$ when $q\notin B$.

The metric $g^\mu(x,t)$ is equal to $g^0(x,t)$ for $t\in(0,\alpha]$; hence $g^\mu(x,t)$ is hyperbolic for $t \in (0,\alpha]$. Also, $g^\mu(x,t)$ 
is the push-forward (by $\phi^\mu$) of the hyperbolic metric $g^0$ for $t\in [\alpha, 2\alpha]$; hence $g^\mu(x,t)$ is hyperbolic for $t\in[\alpha,2\alpha]$.
For $t\in[2\alpha,3\alpha]$, the metric $g^\mu(x,t)$ is similar to the ones constructed in \cite{FJ2} or in Theorem 3.1 in \cite{O}.

It can be checked from those references that the sectional curvatures of $g^\mu$ are $\epsilon$ close to $-1$, provided $r$ is large enough. 

Therefore the $f$ that we constructed in subsection \ref{subsec:construct} maps to the metric $f_*g_s$ which can be joined by a path of metrics in 
$\Met^{<0}(M)$ to the basepoint metric $g_s$ on $M$. 

Since $[f]$ maps to $0$ in $\pi_0(\Met^{<0}(M))$ and $[f]\ne0$ in $\pi_0(\text{Diff}_0(M))$ it can be pulled back to a non-zero element in 
$\pi_1(\cT^{<0}(M))$. Hence the proof of our result is complete.

\bigskip

%%%%%%%%%%%%%%%%%%%%%%%%%%%%%%%%

\end{document}